\documentclass[12pt]{amsart}

\usepackage{amsmath,amsthm,amssymb}
\usepackage{enumitem}
\usepackage{tikz}
\usetikzlibrary{positioning}
\usepackage{hyperref}  
\usepackage{amsfonts,amscd}
\usepackage{mathtools}
\usepackage{wasysym} 
\usepackage{graphicx}
\usepackage{psfrag}
\usepackage{datetime}
\usepackage{tikz-cd}  
\usepackage{geometry} 
\usepackage[all]{xy}

\DeclareFontFamily{U}{mathx}{\hyphenchar\font45}
\DeclareFontShape{U}{mathx}{m}{n}{
      <5> <6> <7> <8> <9> <10>
      <10.95> <12> <14.4> <17.28> <20.74> <24.88>
      mathx10
      }{}
\DeclareSymbolFont{mathx}{U}{mathx}{m}{n}
\DeclareFontSubstitution{U}{mathx}{m}{n}
\DeclareMathAccent{\widecheck}{0}{mathx}{"71}

\let\emptyset\varnothing 
\newtheorem{theorem}{Theorem}[section]
\newtheorem{lemma}[theorem]{Lemma}
\newtheorem{proposition}[theorem]{Proposition}

\newtheorem{innercustomthm}{Theorem}
\newenvironment{customthm}[1]
  {\renewcommand\theinnercustomthm{#1}\innercustomthm}
  {\endinnercustomthm}

\theoremstyle{definition}
\newtheorem{definition}[theorem]{Definition}

\newtheorem{example}[theorem]{Example}

\newtheorem{construction}[theorem]{Construction}

\theoremstyle{remark}
\newtheorem{remark}[theorem]{Remark}

\newcommand{\nonconsec}[2]{\prescript{\circlearrowleft}{}{\mathbf{I}_{#1}^{#2}}}
\newcommand{\subs}[3]{\left(\begin{smallmatrix}[#1]\\#2\end{smallmatrix}\right)_{#3}}

\newcommand{\jm}[1]{\textcolor{black}{#1}}   
\newcommand{\jjm}[1]{\textcolor{black}{#1}}   

\newcommand{\nw}[1]{\textcolor{black}{#1}}    
\newcommand{\jdm}[1]{\textcolor{black}{#1}}   
\newcommand{\nnw}[1]{\textcolor{black}{#1}}    
\newcommand{\njw}[1]{\textcolor{black}{#1}}    

\title[Precluster-tilting for type $A^{d}_n\otimes A^{d}_m$]{The combinatorics of tensor products of higher Auslander algebras of type $A$}
\author{Jordan McMahon}
\address{Universit\"at Graz,
Heinrichstrasse 36, A-8010 Graz, Austria}
\email{jordanmcmahon37@gmail.com}
\author{Nicholas J. Williams}
\address{School of Mathematics and Actuarial Science, University of Leicester, University Road, Leicester, LE1 7RH}
\email{njw40@le.ac.uk}

\begin{document}
\begin{abstract}
We consider maximal non-$l$-intertwining collections, which are a higher-dimensional version of the maximal non-crossing collections which give clusters of \njw{Pl\"ucker} coordinates in the Grassmannian coordinate ring, as described by Scott. We extend a method of Scott for producing such collections, which are related to tensor products of higher Auslander algebras of type $A$.
We show that a higher preprojective algebra of the tensor product of two $d$-representation-finite algebras has a $d$-precluster-tilting subcategory.  Finally we relate mutations of these collections to a form of tilting for these algebras.
\end{abstract}

\maketitle
\tableofcontents

\section{Introduction}

Cluster algebras are commutative rings that were introduced by Fomin and Zelevinsky \cite{fz1,fz2}, in order to study total positivity and dual canonical bases in Lie theory. In the most rudimentary case, a triangulation of a convex polygon consists of a maximal collection of non-crossing arcs, and any two such triangulations are related by a series of flips. A triangulation of a convex polygon is an example of a cluster, and each flip of a triangulation of a convex polygon is an example of a mutation; the collection of clusters and their mutations constitute a cluster algebra. For each positive integer $n\geq 4$, the triangulations of a convex $n$-gon govern the cluster algebra of Dynkin type $A_{n-3}$. 

Fomin and Zelevinsky noted that these triangulations of convex $n$-gons described a cluster structure on $\mathbb{C}[\mathrm{Gr}(2,n)]$, the homogeneous coordinate ring of the Grassmannian $\mathrm{Gr}(2,n)$. More explicitly, one may inscribe the integers $\{1,2,\ldots,n\}$ onto the boundary of an $n$-gon, and rewrite each arc of a triangulation as a 2-subset. Two arcs $(a,c)$ and $(b,d)$ are then crossing whenever $a<b<c<d$  under the cyclic ordering. Each 2-subset labels a Pl\"ucker coordinate in $\mathbb{C}[\mathrm{Gr}(2,n)]$, and flips of a triangulation may be described by Pl\"ucker relations.  This construction was extended by Scott \cite{scott} to the  coordinate ring of any Grassmannian $\mathbb{C}[\mathrm{Gr}(k,n)]$. Here two $k$-subsets of $\{1,2,\ldots ,n\}$ are said to be crossing if there exist cyclically ordered elements $s<t<u<v$ where $s,u\in I$, $s,u\notin J$ and $t,v\in J$, $t,v\notin I$. A maximal collection of non-crossing $k$-subsets of $\{1,2,\ldots,n\}$ is a (Grassmannian) cluster.

Grassmannian clusters have seen much study recently; from the representation-theoretic side their categorifications  were studied directly \cite{gls,jks}; through dimer models \cite{bkm}; and through Frobenius versions \cite{pr2,pr} as well as self-injective quivers with potential \cite{hi2,pas2}. More general models studied in relation to Schubert cells can be found in \cite{gl19,kul,ssw}.
The first step in the construction of the cluster structure of Scott is achieved by the following result, which we will generalise. Here $[n]$ denotes the set $\{ 1, \ldots , n \}$. 

\begin{theorem}\cite[Theorem 1]{scott}
Let $\mathbb{T}_{k,n}$ be a snake triangulation of a convex $n$-gon and $2 \leq k \leq n-2$.
Let $\mathcal{C}_{k}(\mathbb{T}_{k,n})$ be the collection of $k$-subsets of $[n]$ expressible as a unique disjoint union $I\sqcup I^\prime$ of intervals $I$ and $I^\prime$ whose beginning points $i$ and $i^\prime$ form a chord $[ii^\prime]$ in the triangulation $\mathbb{T}_{k,n}$.
 Then $\mathcal{C}_{k}(\mathbb{T}_{k,n})$ is a collection of $(k-1)(n-k-1)$ pairwise non-crossing $k$-subsets of $[n]$.
\end{theorem}

Thus $\mathcal{C}_{k}(\mathbb{T}_{k,n})$ defines a cluster in $\mathbb{C}[\mathrm{Gr}(k,n)]$.
On the other hand, Oppermann and Thomas \cite{ot} generalised the cluster structure of triangulations of convex polygons to cyclic polytopes.  Combinatorially, a triangulation of a cyclic polytope is given by maximal-by-size collections of non-intertwining subsets, where, given two $k$-subsets $I=\{i_1,i_2,\ldots,i_l\}$ and $J=\{j_1,j_2,\ldots,j_l\}$, then \emph{$I$ intertwines $J$} if \[i_{1} < j_{1} < i_{2} < \dots < i_{l} < j_{l}.\] While no cluster algebra is formed, such triangulations are related to the representation theory of higher Auslander algebras of Dynkin type $A$ (see Theorem \ref{oppthom} for a precise description). 
It is of much interest in the literature to study generalisations of cyclic polytopes and their triangulations, such as cyclic zonotopes \cite{dkk,galashin} as well as the amplituhedron \cite{ahtt,aht,gl,kw}.
\nnw{We look} to unify the combinatorics of Grassmannian clusters and triangulations of cyclic polytopes through the study of maximal non-$l$-intertwining collections, where $I$ \emph{$l$-intertwines} $J$ if there are increasing $l$-tuples of real numbers $I' \subseteq I$, $J' \subseteq J$ such that:
\begin{itemize}
\item $I'$ intertwines $J'$,
\item $I' \cap J = \emptyset = I \cap J'$.
\end{itemize}

The cluster algebra belonging to the coordinate ring of the Grassmannian $\mathrm{Gr}(k,n)$ is closely related to the representation theory of quivers of type $A_{n-k-1}\otimes A_{k-1}$, and we refer to the introduction of \cite{marsco} for a more detailed discussion. Quivers of type  $A_{n-k-1}\otimes A_{k-1}$ were studied by Keller \cite{keller} in relation to the periodicity conjecture. A construction for mutation of quivers comprised of unit square tiles is found in Section 11 of \cite{marsco}. Such quivers were shown to be in bijection with rectangular clusters in \jjm{\cite{mc,mc1}}. We relate maximal non-$l$-intertwining collections to the representation theory of quivers of type $A^{d}_{n-k-d}\otimes A^{d}_{k-d}$, where $d=l-1$. 

Path algebras of tensor products of Dynkin type $A$ quivers have been of wide interest in the recent literature. Since these algebras are of infinite representation type, it is common to study only particular subcategories of the module category. This is the case for finding 2-cluster-tilting subcategories as in \cite{hi,pas}; or interval representations  \cite{abeny,be2,eh} and, alternatively, thin modules \cite{bbos} for applications to multi-dimensional persistence homology in Topological Data Analysis. Fomin, Pylyavskyy and Shustin conjecture that any two real \jdm{M}orsifications of real isolated plane curve singularities are related by mutation of their associated quivers if and only if they are of the same complex topological type \cite[Conjecture 5.1]{fps}. Specifically, the associated quivers are \jjm{mutation equivalent to a  relation extension of the tensor products of two type $A$ quivers}. This \jjm{is echoed by} one-dimensional persistence homology in Topological Data Analysis in \cite{eh3,mschr} where two zig-zag Morse filtrations are shown to have the same persistent homology. \jdm{Recently, Dyckerhoff, Jasso and Lekili \cite{djl} connect Morsifications of the Lefschetz fibrations of Auroux \cite{aur} to higher Auslander algebras of type $A$. It follows to question whether tensor products of higher Auslander algebras also have a geometric interpretation as a Morsification.}
We generalise the construction of Scott with the following theorem, and refer to Definition \ref{spread} for the definition of a slice.

\begin{customthm}{\ref{maintheorem}}
Let $\mathcal{T}$ be a non-intertwining collection of $\binom{\nw{n-l-1}}{l-1}$  $l$-subsets of $[n]$ that is a slice and $l \leq k \leq \nw{n-l}$\jjm{.} Let $\mathcal{C}_{k}(\mathcal{T})$  be the collection of $k$-subsets of $[n]$ expressible as a unique disjoint union $I_1\sqcup I_2\sqcup\cdots \sqcup I_l$ of intervals $I_1,\ldots,I_l$  whose beginning points $i_1,\ldots,i_l$ determine a member $(i_1,i_2,\ldots,i_l)$ of $\mathcal{T}$. 
Then $\mathcal{C}_{k}(\mathcal{T})$ is non-$l$-intertwining collection of $\binom{k-1}{l-1}\binom{\nw{n-k-1}}{l-1}$ increasing $k$-tuples in $[n]$.
\end{customthm}

Higher Auslander--Reiten theory was introduced by Iyama \cite{iy3,iy1} and is an active area of current research, see for example \jjm{\cite{di, djw, djl,fed,hio,hjv,jasso,jkv,jj1}}.
We study higher preprojective algebras of tensor products of higher Auslander algebras of type $A$. These algebras often have infinite global dimension, and hence provide useful examples for studying singularity categories, as in \cite{ahs,kvamme2,mc2,mc4}. Current examples of algebras with infinite global dimension that are not self-injective and appear in higher Auslander--Reiten theory are rather sparse, largely  limited to higher Nakayama algebras \cite{jkue} and cluster-tilted algebras \cite{ot}.
Another reason to study tensor products of higher Auslander algebras of type $A$, is that the resulting algebras are often quadratic. This is related to the property of being Koszul, and algebras with this property have been studied in connection to higher Auslander--Reiten theory in \cite{bocca,grant,gi}.  We are able to construct \nnw{$d$}-precluster-tilting subcategories, as introduced by Iyama and Solberg \cite{is}. \jdm{These} \nnw{$d$}-precluster-tilting subcategories are a weaker version of higher cluster-tilting subcategories, but which are of interest in their own right since they (also) correspond to algebras with large dominant dimension that have interesting tilting modules, as studied in \cite{lizhang,marcz,nrtz,ps} (see also \cite{chenxi}).

\begin{customthm}{\ref{precl}}
Suppose $A$ and \jm{$B$} are $d$-representation-finite \jm{$K$-}algebras and define $\Lambda=A\jm{^\mathrm{op}}\otimes\jm{{_K}B}$. Then there exists a $d$-precluster-tilting subcategory $$\mathcal{D}\subseteq\mathrm{mod}(\Pi_{2d+1}(\Lambda)).$$
\end{customthm}

\jjm{In the final section we indicate how a form of the $d$-APR-tilting of Iyama and Oppermann \cite{io} is able to model the combinatorics of mutations of maximal non-$l$-intertwining collections. We note that the confluence of quadratic algebras and APR-tilting has been studied in \cite{guoxiao,guoxiao2}.}

\section*{Acknowledgements}
\jdm{J.M. was partially supported by the Austrian Science Fund (FWF): W1230. 
N.J.W.} would like to thank Professor Karin Baur for hosting him in Graz, where some of the work behind this paper was done. He is also grateful to the University of Leicester for funding his PhD through support of a grant \jdm{of his} supervisor Professor Sibylle Schroll from the EPSRC (reference EP/P016294/1).

\section{Background}

\nw{For $i \in [m]$ we use $<_{i}$ to denote the cyclically shifted order on $[m]$ given by:
\[i <_{i} i+ 1 <_{i} < \dots <_{i} m-1 <_{i} m <_{i} 1 <_{i} \dots <_{i} i-1 .\] For $r \geq 3$, $a_{1} < \dots < a_{r}$ is a \emph{cyclic ordering} if there is an $i \in [m]$ such that $a_{1} <_{i} \dots <_{i} a_{r}$.} We denote by $(a,b) \subseteq [n]$ the open \nw{cyclic} interval and use $[a,b] \subseteq [n]$ to denote the closed \nw{cyclic} interval. Subsets of the form $[a,b]$ are called \emph{interval subsets}.

\subsection{Grassmannian cluster algebras}

\jjm{We refer to the book \cite{marshbook} for an introduction to cluster algebras from Grassmannians.} 
Recall that the Grassmannian of all $k$-dimensional subspaces of $\mathbb{C}^n$, $\mathrm{Gr}(k,n)$, can be embedded into the projective space $\mathbb{P}(\bigwedge^k(\mathbb{C}^n))$ via the Pl\"ucker embedding. The coordinates of $\bigwedge^k(\mathbb{C}^n)$ are called \emph{Pl\"ucker coordinates} and are indexed by the $k$-subsets $I=\{i_1,i_2,\ldots,i_k\}\subset{[n]},$ where $1\leq i_1< i_2< \cdots< i_k\leq n$.

Two $k$-subsets $I$ and $J$ are said to be \emph{non-crossing} (also referred to as \emph{weakly separated} in some articles such as \cite{lz,os,ops,scott-quant}
) if there do not exist elements $s<t<u<v$ (ordered cyclically) where $s,u\in I-J$, and $t,v\in J-I$. 

\begin{theorem}\cite[Theorem 3]{scott}
There is a cluster structure on the coordinate ring $\mathbb{C}[\mathrm{Gr}(k,n)]$. Each maximal collection of pairwise non-crossing $k$-subsets of $\{1,2,\ldots,n\}$ determines a cluster in the structure.
\end{theorem}

\begin{theorem}\cite[Theorem 3]{scott}\cite[Theorem 1.6]{ops}\cite{dkk10}\label{clusterbij}
There are bijections between:
\begin{itemize}
\item Maximal collections of pairwise non-crossing $k$-subsets of $[n]$,
\item Collections of $(k-1)(n-k-1)$ pairwise non-crossing, non-interval $k$-subsets of $[n]$,
\item Clusters consisting only of Pl\"ucker coordinates in the cluster structure of the  coordinate ring $\mathbb{C}[\mathrm{Gr}(k,n)]$.
\end{itemize}
\end{theorem}

\nnw{An introduction to cluster algebras from Grassmannians may be found in \cite{marshbook}.}


\subsection{Triangulations of cyclic polytopes}\label{cyclic-poly}

\begin{definition} 
Given two increasing $l$-tuples of real numbers $I=\{i_1,i_2,\ldots,i_l\}$ and $J=\{j_1,j_2,\ldots,j_l\}$, \emph{$I$ intertwines $J$} if $$i_{1} < j_{1} < i_{2} < \dots < i_{l} < j_{l}.$$
In this case, we write $I \wr J$. Unless otherwise specified, we will henceforth assume all $l$-tuples are increasing. A collection of $l$-tuples of real numbers is \emph{non-intertwining} if no pair of elements intertwine (in either order). In this paper we generally work modulo $n$ with cyclic ordering. \jjm{We will write  $I \wr J$ whenever $I$ and $J$ are intertwining in either order.}
\end{definition}
As in Definition 2.2 of \cite{ot}, we consider the \nw{collections}.
\begin{align*}
\mathbf{I}_{n}^{l-1} &:= \{ (i_{1}, \dots , i_{l}) \in [n]^{l} \mid \forall x \in \{ 1, 2, \dots ,\nw{l-1} \}, i_{x+1}\geq i_x+2 \} \\
\nw{\nonconsec{n}{l-1}} &\nw{:= \{ (i_{1}, \dots , i_{l}) \in \mathbf{I}_{n}^{l-1} \mid i_{l}+2 \leq i_{1}+n \}} 
\end{align*}
We omit the definition of a cyclic polytope, and refer the reader to \cite[Section 2]{ot} for more details.

\begin{theorem}\cite[\nnw{Theorem 1.1}]{ot}
There is a bijection between:
\begin{itemize}
\item Triangulations of the cyclic polytope $\mathrm{C}(\nw{n},2l-2)$,
\item Collections of $\binom{n-l}{l-1}$ pairwise non-intertwining $l$-tuples in $\mathbf{I}_{n}^{l-1}$.
\end{itemize}
\end{theorem}

Furthermore Lemma 2.20 of \cite{ot} implies that $\binom{n-l}{l-1}$ is the maximal size of a non-intertwining collection of elements in $\mathbf{I}_n^{l-1}$.
 Elements of $\mathbf{I}_{n}^{l-1} \setminus \nonconsec{n}{l-1}$ cannot intertwine with any other element of $\mathbf{I}_{n}^{l-1}$, so a maximal non-intertwining collection of $l$-tuples from $\mathbf{I}_{n}^{l-1}$ consists of $\mathbf{I}_{n}^{l-1} \setminus \nonconsec{n}{l-1}$ along with a maximal non-intertwining collection of $l$-tuples from $\nonconsec{n}{l-1}$. \jjm{Compare this with the discussion at the beginning of \cite[Section 8]{ot}.} Since $|\mathbf{I}_{n}^{l-1} \setminus \nonconsec{n}{l-1}| = \binom{n-l-1}{l-2}$, we obtain that the maximal size of a non-intertwining collection of $l$-tuples from $\nonconsec{n}{l-1}$ is $\binom{n-l-1}{l-1} = \binom{n-l}{l-1} - \binom{n-l-1}{l-2}$. The following formulation holds.

\begin{theorem}\cite[\nnw{Theorem 1.2}]{ot}\label{oppthom}
There is a bijection between:
\begin{itemize}
\item Triangulations of the cyclic polytope \jjm{$\mathrm{C}(n,\nnw{2l-2)}$},
\item \nw{Collections of $\binom{n-l-1}{l-1}$ pairwise non-intertwining $l$-tuples in $\nonconsec{n}{l-1}$.}
\end{itemize}
\end{theorem} 



\section{Non-$l$-intertwining collections}

Non-intertwining collections may be generalised in the following way, which we recall from the introduction. 

\begin{definition}
Given an integer $l$ such that $2 \leq l \leq k$ and two increasing $k$-tuples of real numbers $I, J$, then we say $I$ \emph{$l$-intertwines} $J$ if there are increasing $l$-tuples of real numbers $I' \subseteq I$, $J' \subseteq J$ such that:
\begin{itemize}
\item $I'$ intertwines $J'$.
\item $I' \cap J = \emptyset = I \cap J'$.
\end{itemize}
 In the literature, $l$-intertwining subsets are sometimes called $(2l-2)$-separated \cite{dkk} or $l$-interlacing \cite{bbge}. 
Given a collection of $k$-subsets $\mathcal{C} \subseteq \subs{n}{k}{}$, we say that $\mathcal{C}$ is non-$l$-intertwining if no pair of elements $l$-intertwine (in either order).
\end{definition}
If $I\subseteq [n]$ is a $k$-subset with unique decomposition $I = I_{1} \sqcup \dots \sqcup I_{l}$ where $I_{i}$ are intervals then we call $I$ an \emph{$l$-ple interval}.
We denote the set of $k$-subsets of $[n]$ that are $l$-ple intervals by $\subs{n}{k}{l}$.
Given an $l$-ple interval \[ I = [t_{1},t'_{1}] \sqcup \dots \sqcup [t_{l},t'_{l}],\] we denote
\begin{align*}
\widehat{I} &:= \{ t_{1}, \dots ,t_{l} \}, \\
\widecheck{I} &:= \{ t'_{1}, \dots ,t'_{l} \}.
\end{align*}
We wish to construct maximal non-$l$-intertwining \nw{collections of $l$-ple intervals from $[n]$.} 

\begin{theorem}\cite[Theorem 1]{scott}\label{scott-thm-1}
Let $\mathbb{T}_{k,n}$ be a snake triangulation of a convex $n$-gon, with $2 \leq k \leq n-2$. Define \[ \mathcal{C}_{k}(\mathbb{T}_{k,n}) := \{ I \in \subs{n}{k}{2} \mid \widehat{I} \in \mathbb{T}_{k,n} \}.\]  Then $\mathcal{C}_{k}(\mathbb{T}_{k,n})$ is a maximal collection of pairwise non-crossing, non-interval $k$-subsets of $[n]$.
\end{theorem}

We refer to Figure \nw{\ref{snake}} for an example of a snake triangulation. In the terminology of \cite{scott} a snake triangulation is a ``zig-zag'' triangulation.
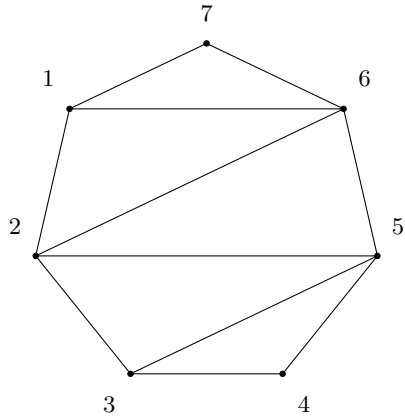
\begin{figure}
\caption{A snake triangulation of a heptagon}\label{snake}
\[\begin{tikzpicture}[scale=2]
\draw (0.,4.)-- (-1.,4.);
\draw (-1.,4.)-- (-1.6234898018587332,4.7818314824680295);
\draw (-1.6234898018587332,4.7818314824680295)-- (-1.4009688679024186,5.756759394649853);
\draw (-1.4009688679024186,5.756759394649853)-- (-0.5,6.19064313376741);
\draw (-0.5,6.19064313376741)-- (0.40096886790241903,5.756759394649852);
\draw (0.40096886790241903,5.756759394649852)-- (0.6234898018587329,4.781831482468029);
\draw (0.6234898018587329,4.781831482468029)-- (0.,4.);
\draw (-1.4009688679024186,5.756759394649853)-- (0.40096886790241903,5.756759394649852);
\draw (0.40096886790241903,5.756759394649852)-- (-1.6234898018587332,4.7818314824680295);
\draw (-1.6234898018587332,4.7818314824680295)-- (0.6234898018587329,4.781831482468029);
\draw (0.6234898018587329,4.781831482468029)-- (-1.,4.);
\begin{scriptsize}
\draw [fill=black] (-1.,4.) circle (0.5pt);
\draw[color=black] (-1.14,3.8) node {3};
\draw [fill=black] (0.,4.) circle (0.5pt);
\draw[color=black] (0.14,3.8) node {4};
\draw [fill=black] (0.6234898018587329,4.781831482468029) circle (0.5pt);
\draw[color=black] (0.76,4.98) node {5};
\draw [fill=black] (0.40096886790241903,5.756759394649852) circle (0.5pt);
\draw[color=black] (0.54,5.96) node {6};
\draw [fill=black] (-0.5,6.19064313376741) circle (0.5pt);
\draw[color=black] (-0.5,6.39) node {7};
\draw [fill=black] (-1.4009688679024186,5.756759394649853) circle (0.5pt);
\draw[color=black] (-1.54,5.96) node {1};
\draw [fill=black] (-1.6234898018587332,4.7818314824680295) circle (0.5pt);
\draw[color=black] (-1.76,4.98) node {2};
\end{scriptsize}
\end{tikzpicture}\]
\end{figure}
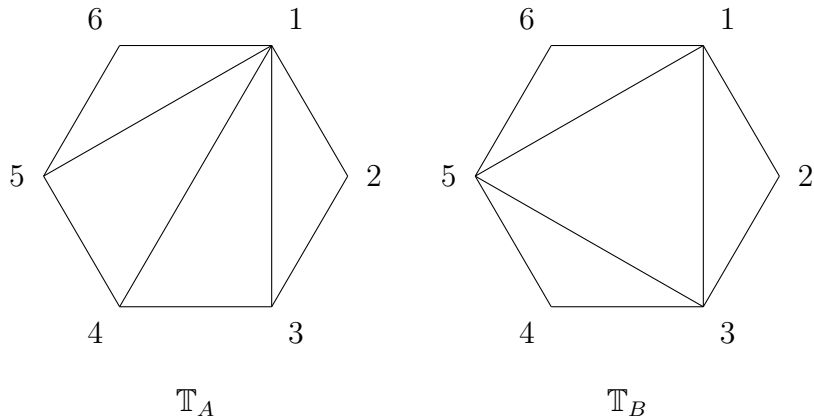
\begin{figure}[h]
\caption{Example triangulations}\label{triangs}
\[
\begin{tikzpicture}
\coordinate(2) at (0:2);
\node at (2) [right = 1mm of 2] {2};
\coordinate(1) at (60:2);
\node at (1) [above right = 1mm of 1] {1};
\coordinate(6) at (120:2);
\node at (6) [above left = 1mm of 6] {6};
\coordinate(5) at (180:2);
\node at (5) [left = 1mm of 5] {5};
\coordinate(4) at (240:2);
\node at (4) [below left = 1mm of 4] {4};
\coordinate(3) at (300:2);
\node at (3) [below right = 1mm of 3] {3};

\draw (1) -- (2);
\draw (2) -- (3);
\draw (3) -- (4);
\draw (4) -- (5);
\draw (5) -- (6);
\draw (6) -- (1);

\draw (1) -- (3);
\draw (1) -- (4);
\draw (1) -- (5);

\node(A) at (270:3) {$\mathbb{T}_A$};
\end{tikzpicture}
\hspace*{0.5cm}
\begin{tikzpicture}
\coordinate(2) at (0:2);
\node at (2) [right = 1mm of 2] {2};
\coordinate(1) at (60:2);
\node at (1) [above right = 1mm of 1] {1};
\coordinate(6) at (120:2);
\node at (6) [above left = 1mm of 6] {6};
\coordinate(5) at (180:2);
\node at (5) [left = 1mm of 5] {5};
\coordinate(4) at (240:2);
\node at (4) [below left = 1mm of 4] {4};
\coordinate(3) at (300:2);
\node at (3) [below right = 1mm of 3] {3};

\draw (1) -- (2);
\draw (2) -- (3);
\draw (3) -- (4);
\draw (4) -- (5);
\draw (5) -- (6);
\draw (6) -- (1);

\draw (1) -- (3);
\draw (3) -- (5);
\draw (1) -- (5);

\node(B) at (270:3) {$\mathbb{T}_B$};
\end{tikzpicture}
\]
\end{figure}
\begin{example}\label{not-all-triangs}
The construction in Theorem \ref{scott-thm-1} can be generalised, but \jjm{it} is not always possible.
In Figure \ref{triangs}, triangulation $\mathbb{T}_A$ has chords $\{13, 14, 15 \}$. Then $\mathcal{C}_{3}(\mathbb{T}_A)$ is a maximal  collection of pairwise non-crossing, non-interval 3-subsets of $[6]$ given by $\{ 134, 124, 145, 125 \}$.
In Figure \ref{triangs}, triangulation $\mathbb{T}_B$ has chords $\{ 13, 15, 35 \}$. Although $\mathcal{C}_{3}(\mathbb{T}_A)$ is a collection of pairwise non-crossing, non-interval $3$-subsets of $[6]$ given by $\{ 134, 356, 125 \}$, it is not maximal.
\end{example}
For $l \geq 2$, we consider triangulations of cyclic polytopes, as described in Section \ref{cyclic-poly}. As seen in Example \ref{not-all-triangs}, it is not possible to construct maximal non-$l$-intertwining collections from all triangulations using $\mathcal{C}_k$. Hence we must refine the class of triangulations of cyclic polytopes we are considering.

\begin{definition}\label{phi}
For positive integers $n$ and $l$, we define \nw{functions}
\begin{align*}
\phi\nw{_{i}}: \nonconsec{n}{l-1} &\rightarrow \mathbb{Z}_{ \geq 0}^{l} \\
(t_{1} , \dots , t_{l}) &\mapsto (t_{2}-t_{1}-2, t_{3}-t_{2}-2, \dots , t_{l}-t_{l-1}-2, t_{1}-t_{l}\nw{-2}).
\end{align*}
\nw{where $t_{1}<_{i} \dots <_{i} t_{l}$ and arithmetic is modulo $n$}.
For a non-intertwining collection $\mathcal{T}$ of $\binom{\nw{n-l-1}}{l-1}$ elements in $\nonconsec{n}{l-1}$, we denote \[\phi\nw{_{i}}(\mathcal{T}) = \{ \phi\nw{_{i}}(T) \mid T \in \mathcal{T} \}.\]
\end{definition}

\begin{definition}\label{spread}
A non-intertwining collection of $\binom{\nw{n-l-1}}{l-1}$ elements in $\nonconsec{n}{l-1}$ is a \emph{slice} if \nw{there is an $i \in [n]$ such that} \[\phi\nw{_{i}}(\mathcal{T})=\{ (p_{1}, \dots , p_{l}) \in \mathbb{Z}_{\geq 0}^{l} \mid \sum_{j=1}^{l}p_{j}=\nw{n-2l} \}.\]
\end{definition}
This definition is intended to mirror the notion of a slice in \cite[Definition 4.8]{io}.
\begin{example}\label{ex-spread}
The non-intertwining collection of $\binom{4}{2}$ elements in $\nonconsec{8}{2}$ $$\mathcal{T}_1 = \{ 135, 136, 137, 146, 147, 157 \}$$ is a slice: $\phi\nw{_{1}}(\mathcal{T}_1)=\{ 002, 011, 020, 101, 110, 200 \}$. On the other hand, the non-intertwining collection of $\binom{4}{2}$ elements in $\nonconsec{8}{2}$ $$\mathcal{T}_2 = \{ 135, 136, 137, 357, 147, 157 \}$$ is not a slice: \nw{$\phi_{i}(135)=\phi_{i}(357)$ for $i \in \{8,1\}$, $\phi_{i}(357)=\phi_{i}(571)$ for $i \in \{2,3\}$, $\phi_{4}(471)=\phi_{4}(613)$, $\phi_{5}(713)=\phi_{5}(571)$, $\phi_{i}(713)=\phi_{i}(135)$ for $i \in \{6,7\}$}. 
\end{example}
We now characterise when two \emph{$l$-ple} intervals are $l$-intertwining.
\begin{lemma}\label{int-endpoints-lemma}
Let  $I,J$ be two $k$-subsets of $[n]$ that are $l$-ple intervals with $|I| = |J| = k \geq l$. Then $I$ and $J$ are $l$-intertwining if and only if both $\widehat{I}\wr\widehat{J}$ and  $\widecheck{I}\wr \widecheck{J}$.
\end{lemma}
\begin{proof}
Let  $I = [i_{1}, i'_{1}] \sqcup \dots \sqcup [i_{l}, i'_{l}]$, $J = [j_{1}, j'_{1}] \sqcup \dots \sqcup [j_{l}, j'_{l}]$. First suppose $I$ and $J$ are $l$-intertwining; we have $l$-subsets $A \subseteq I$ and $B \subseteq J$ such that $A \cap J = \emptyset = B \cap I$ and that $A$ and $B$ are intertwining. Hence let $A = \{ a_{1}, \dots , a_{l} \}$ and $B = \{ b_{1}, \dots , b_{l} \}$ such that \[ a_{1} < b_{1} < a_{2} < \dots < a_{l} < b_{l}.\] For any $ 1\leq r < s\leq l$ and some $1\leq t\leq l$, we cannot have both $a_{r}, a_{s} \in [i_{t}, i'_{t}]$, since this implies that $b_{r}, \dots, b_{s-1} \in [i_{t}, i'_{t}] \subseteq I$, but $B \cap I = \emptyset$. Hence, without loss of generality, we may assume that $a_{t} \in [i_{t}, i'_{t}]$ for all $1\leq t\leq l$. Similarly, we reason that $b_{t} \in [j_{t}, j'_{t}]$ for all $1\leq t\leq l$. Moreover, since $B \cap I = \emptyset$, we have $b_{t} \in (i'_{t}, i_{t+1})$, and similarly $a_{t} \in (j'_{t-1}, j_{t})$. Then \[ i_{1} \leq a_{1} < j_{1} \leq b_{1} < i_{2} \leq a_{2} < \dots  < i_{l} \leq a_{l} < j_{l} \leq b_{l},\] and so \[ i_{1} < j_{1} < i_{2} < \dots < i_{l} < j_{l}.\] In other words, $\widehat{I}\wr \widehat{J}$. The same reasoning implies  $\widecheck{I}\wr \widecheck{J}$.

Conversely, assume $(\widehat{I}, \widehat{J})$ and $(\widecheck{I}, \widecheck{J})$ are both intertwining pairs of subsets and $i_{1} < j_{1} < i_{2} < \dots < i_{l} < j_{l}$. Then we must have $i'_{1} < j'_{1} < i'_{2} < \dots < i'_{l} < j'_{l}$, otherwise every interval of $I$ strictly contains an interval of $J$, contradicting $|I| = |J|$. For each $1\leq s\leq l$, $j'_{s-1} < j_{s}$, $i_{s} < j_{s}$ and  $j'_{s-1} < i'_{s}$, so $[i_{s}, i'_{s}] \cap (j'_{s-1}, j_{s}) \neq \emptyset$. Hence pick $a_{s} \in [i_{s}, i'_{s}] \cap (j'_{s-1}, j_{s})$ and, similarly, pick $b_{s} \in [j_{s}, j'_{s}] \cap (i'_{s}, i_{s+1})$ for each $1\leq s \leq l$. Let $A = \{ a_{1}, \dots ,  a_{l} \}$ and $ B = \{ b_{1}, \dots, b_{l} \}$. Then, by construction, $A \subseteq I - J$, $B\subseteq J - I$, $|A| = |B| = l$ and $A$ and $B$ are intertwining. Therefore $I$ and $J$ are $l$-intertwining.
\end{proof}

Now we describe our first main construction.

\begin{theorem}\label{maintheorem}
Let $\mathcal{T}$ be a non-intertwining collection of $\binom{\nw{n-l-1}}{l-1}$ elements in $\nonconsec{n}{l-1}$ that is a slice and $l\leq k\leq \nw{n-l}$. Define the collection \[ \mathcal{C}_{k}(\mathcal{T}) := \{ I \in \subs{n}{k}{l} \mid \widehat{I} \in \mathcal{T} \}.\] Then $\mathcal{C}_{k}(\mathcal{T})$ is a collection of $\binom{k-1}{l-1}\binom{\nw{n-k-1}}{l-1}$ non-$l$-intertwining $k$-tuples. 
Dually, the collection \[ \mathcal{C}'_{k}(\mathcal{T}) := \{ I \in \subs{n}{k}{l} \mid \widecheck{I} \in \mathcal{T} \} \] is a collection of $\binom{k-1}{l-1}\binom{\nw{n-k-1}}{l-1}$ non-$l$-intertwining $k$-tuples. 
\end{theorem}

\begin{proof}
Let $\mathcal{T}$ be a non-intertwining collection of $\binom{\nw{n-l-1}}{l-1}$ elements in $\nonconsec{n}{l-1}$ that is a slice. 
Let \[\mathcal{S}:= \{ ((q_{1}, \dots , q_{l}),(r_{1}, \dots , r_{l})) \in \mathbb{Z}_{\geq 0}^{l}\times \mathbb{Z}_{\geq 0}^{l} \mid \sum_{j=1}^{l}q_{j}=k-l, \sum_{j=1}^{l}r_{j}=\nw{n-k-l} \}.\] 
\nw{Then} $|\mathcal{S}|=\binom{k-1}{l-1}\binom{\nw{n-k-1}}{l-1}$.
There is a natural map
\begin{align*}
\mathcal{S} &\rightarrow  \mathbb{Z}_{\geq 0}^{l}\\
((q_{1}, \dots , q_{l}),(r_{1}, \dots , r_{l})) &\mapsto (q_{1}+r_{1}, \dots , q_{l}+r_{l}),
\end{align*} with image $$\{ (p_{1}, \dots , p_{l})\in  \mathbb{Z}_{\geq 0}^{l} \mid \sum_{j=1}^{l}p_{j}=n-2l+1 \}.$$
Since $\mathcal{T}$ is a slice, this  image is precisely the image $\phi\nw{_{i}}(\mathcal{T})$ \nw{for some $i$}. Hence there is a well-defined inverse of $\phi\nw{_{i}}$. So let $(t_{1}, \dots ,t_{l})=\phi\nw{_{i}}^{-1}(q_{1}+r_{1}, \dots , q_{l}+r_{l})$. Consider the map \nw{$\psi: \mathcal{S} \rightarrow \mathcal{C}_{k}(\mathcal{T})$,}
\begin{align*}
\psi:((q_{1}, \dots , q_{l})(r_{1}, \dots , r_{l})) &\mapsto [t_{1}, t_{1}+q_{1}] \sqcup \dots \sqcup [t_{l}, t_{l}+q_{l}].
\end{align*}
By definition $\phi\nw{_{i}}(t_1,\ldots,t_l)=(t_2-t_1-2,\ldots,t_l-t_{l-1}-2,t_1-t_l\nw{-2})$. Then for all $1\leq j\leq l$ we have $t_{j+1}-t_j-2=q_i+r_i$ and so $(t_{j+1}-(t_j+ q_j))\geq2 $. Hence $\psi$ has image $\mathcal{C}_k$ \jdm{and an appropriate choice of $q_1,\ldots,q_l$ describes any element of $\mathcal{C}_k$ as an image of $\psi$}.  Moreover, \jdm{no two elements of $\mathcal{S}$ can have the same image under $\psi$. Therefore $\psi$ is bijective and} $|\mathcal{C}_{k}(\mathcal{T})|=|\mathcal{S}|=\binom{k-1}{l-1}\binom{\nw{n-k-1}}{l-1}$. 
That $\mathcal{C}_{k}(\mathcal{T})$ is non-$l$-intertwining follows from Lemma \ref{int-endpoints-lemma}. A similar proof applies to $\mathcal{C}'_{k}(\mathcal{T})$.
\end{proof}

\begin{example}\label{ex-construction}
Following on from Example \ref{ex-spread}, the non-intertwining collection of $\binom{4}{2}$ elements in $\nw{\nonconsec{8}{2}}$ $$\mathcal{T}_1 = \{ 135, 136, 137, 146, 147, 157 \}$$ extends to the non-$3$-intertwining collection of $\binom{4-1}{3-1}\binom{\nw{8-4-1}}{3-1}=\binom{3}{2}\binom{3}{2}=3 \times 3=9$ elements: $$\mathcal{C}_{4}(\mathcal{T}_1) = \{ 1356, 1346, 1367, 1347, 1246, 1467, 1247, 1457, 1257 \}.$$ 
\end{example}


We expect that the maximal size of a non-$l$-intertwining collection $\mathcal{C} \subset \subs{n}{k}{l}$ is $\binom{k-1}{l-1}\binom{\nw{n-k-1}}{l-1}$, although this appears to be an open question in the literature. Non-$l$-intertwining collections have been studied to some extent as $(2l-2)$-separated collections  in \cite{dkk,dkk3}. However these works are concerned with collections of subsets where subsets of all sizes are permitted, rather than fixing a size $k$. 
These results are based on studies of oriented matroids \cite{gp} and cyclic zonotopes \cite{galashin}. 

\section{\jdm{Higher} precluster-tilting subcategories}\label{precl-sect}

Consider a finite-dimensional algebra $\nnw{\Lambda}$ over a field $K$, and fix a positive integer $d$. We will assume that $\nnw{\Lambda}$ is of the form $KQ/I$, where $KQ$ is the path algebra over some quiver $Q$ and $I$ is an admissible ideal of $KQ$. For two arrows in $Q$, $\alpha:i\rightarrow j$ and $\beta:j\rightarrow k$, we denote their composition as $\beta\alpha:i\rightarrow k$. Let $\nnw{\Lambda}^\mathrm{op}$ denote the opposite algebra of $\nnw{\Lambda}$.
An $\nnw{\Lambda}$-module will mean a finitely-generated left $\nnw{\Lambda}$-module; by $\mathrm{mod}(\nnw{\Lambda})$ we denote the category of $\nnw{\Lambda}$-modules. The functor $D=\mathrm{Hom}_K(-,K)$ defines a duality; by $\otimes$ we mean $\otimes_K$ and we denote the syzygy by $\Omega$. \jjm{Denote by $\nu:= D\nnw{\Lambda}\otimes_{\nnw{\Lambda}}-\cong D\mathrm{Hom}_{\nnw{\Lambda}}(-,\nnw{\Lambda})$ the Nakayama functor in $\mathrm{mod}(\nnw{\Lambda})$}. Let $\mathrm{add}(M)$ be the full subcategory of $\mathrm{mod}(\nnw{\Lambda})$ composed of all $\nnw{\Lambda}$-modules isomorphic to direct summands of finite direct sums of copies of $M$. 
\begin{definition}\cite[Definition 2.2]{iy1}
For a finite-dimensional algebra $\nnw{\Lambda}$, a module $M\in\mathrm{mod}(\nnw{\Lambda})$ is a \emph{$d$-cluster-tilting module} if it satisfies the following conditions:
\begin{align*}
\mathrm{add}(M)&=\{X\in\mathrm{mod}(\nnw{\Lambda})\mid \mathrm{Ext}^i_{\nnw{\Lambda}}(M,X)=0\ \forall\ 0< i<d \}.\\
\mathrm{add}(M)&=\{X\in\mathrm{mod}(\nnw{\Lambda}) \mid \mathrm{Ext}^i_{\nnw{\Lambda}}(X,M)=0\ \forall\ 0< i<d \}.
\end{align*} 
In this case $\mathrm{add}(M)$ is a \emph{$d$-cluster-tilting subcategory} of $\mathrm{mod}(\nnw{\Lambda})$.
\end{definition}
If $\nnw{\Lambda}$ is a finite-dimensional algebra such that there exists a $d$-cluster-tilting module $M\in \mathrm{mod}(\nnw{\Lambda})$ as well as $\mathrm{gl.dim}(\nnw{\Lambda})\leq d$, then $\nnw{\Lambda}$ is \emph{$d$-representation finite}. 

\begin{remark}
An observation that we  frequently use is that an algebra $\nnw{\Lambda}$ is $d$-representation finite if and only if its opposite algebra $\nnw{\Lambda}^\mathrm{op}$ is also $d$-representation finite. This follows from applying the standard duality $D$ to the categories \jjm{$\mathrm{mod}(\nnw{\Lambda})$ and $\mathrm{mod}(\nnw{\Lambda}^\mathrm{op})$.}
\end{remark}
Define  $\tau_d:=\tau\Omega^{d-1}$ to be the \emph{$d$-Auslander--Reiten translation} and  $\tau_{d}^-:=\nw{\tau^{-}\Omega^{-(d-1)}}$ to be the \emph{inverse $d$-Auslander--Reiten translation}.

\begin{theorem}\cite[Theorem 1.6]{iy2}\label{d-clus-tilt}
Let $\Lambda$ be a $d$-representation-finite algebra. Then $\Lambda$ has a unique $d$-cluster-tilting module $M$ and
\[ \mathrm{add}M := \{ \tau_{d}^{-i}\Lambda \mid i \in \mathbb{N} \} = \{ \tau_{d}^{i}D\Lambda \mid i \in \mathbb{N} \}.\]
\end{theorem}


In the context of generalising classical Auslander--Reiten theory to higher dimensions, Iyama introduced in \cite{iy2} the notion of a higher Auslander algebra. For a quiver $Q$ of type $A_n$, there is an explicit description of the \jjm{$d$-Auslander algebra $A^d_n$} of the path algebra $KQ$, \jjm{a $d$-representation-finite algebra. The category $\mathrm{mod}(A^{d}_n)$ has a unique $d$-cluster-tilting subcategory by Theorem \ref{d-clus-tilt}.}

\jjm{\begin{construction}\label{ARdefn}\jdm{\cite{iy1,ot}} Let $A_n^{1}$ be the following quiver of Dynkin type $A_n$. 
\[ \jjm{1 \rightarrow 2 \rightarrow \cdots \rightarrow n} \]
Denote by $A_n^{d}$ the $(d-1)$-Auslander algebra of $A_n^{d-1}$, described as follows.
Let $Q^{d}_{n}$ be the quiver with vertices indexed by the elements of $\mathbf{I}_{n+2d-2}^{d-1}$ and arrows $\alpha_{i}(I):I\rightarrow J$ wherever $I\setminus\{i\}=J\setminus\{i+1\}$ for some $i\in I$. Let $I^{d}_n$ be the admissible ideal of $\nnw{K}Q^{d}_n$ generated by the elements $$\alpha_j(\alpha_i(I))-\alpha_i(\alpha_j(I)),$$ which range over the elements of $\mathbf{I}_{n+2d-2}^{d-1}$. By convention, $\alpha_i(I)=0$ whenever $I$ or $I\cup\{i+1\}\setminus \{i\}$ is not a member of $\mathbf{I}_{n+2d-2}^{d-1}$\nnw{. H}ence there are zero relations included in the ideal $I^{d}_n$. 
\end{construction} }

\jjm{The algebra of $A^2_4$ is as follows, with dotted arrows indicating relations:}

$$
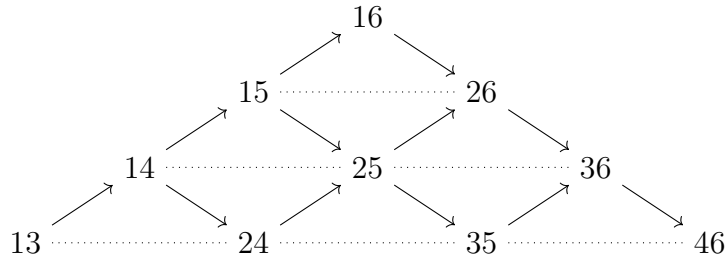
\begin{tikzpicture}[xscale=7.5,yscale=5]
\node(xa) at (-2,1){$13$};
\node(xb) at (-1.6,1){$24$};
\node(xc) at (-1.2,1){$35$};
\node(xd) at (-0.8,1){$46$};

\node(xab) at (-1.8,1.2){$14$};
\node(xac) at (-1.4,1.2){$25$};
\node(xad) at (-1,1.2){$36$};

\node(xbb) at (-1.6,1.4){$15$};
\node(xbc) at (-1.2,1.4){$26$};

\node(xcc) at (-1.4,1.6){$16$};

\draw[->](xa) edge(xab);
\draw[->](xb) edge(xac);
\draw[->](xc) edge(xad);

\draw[->](xbb) edge(xac);
\draw[->](xbc) edge(xad);

\draw[->](xbb) edge(xcc);

\draw[->](xab) edge(xb);
\draw[->](xac) edge(xc);
\draw[->](xad) edge(xd);

\draw[->](xab) edge(xbb);
\draw[->](xac) edge(xbc);

\draw[->](xcc) edge(xbc);

\draw[-,dotted](xa)edge (xb);
\draw[-,dotted](xb)edge (xc);
\draw[-,dotted](xc)edge (xd);

\draw[-,dotted](xab)edge (xac);
\draw[-,dotted](xac)edge (xad);

\draw[-,dotted](xbb)edge (xbc);
\end{tikzpicture}$$

\jjm{A combinatorial description of the $d$-cluster-tilting subcategory $\mathcal{C}\subseteq \mathrm{mod}(A^d_n)$ is possible.}
\jjm{\begin{theorem}\cite[Theorem 3.6]{ot}\label{oppthom2}
The objects of the $d$-cluster-tilting subcategory $\mathcal{C}\subseteq \mathrm{mod}(A^d_n)$ may be indexed by $\mathbf{I}_{n+2d}^{d}$. For each $I\in\mathbf{I}_{n+2d}^{d}$, denote by $M_I$ the object of this $d$-cluster-tilting subcategory.
\begin{itemize}
\item The module $M_I$ is projective if and only if $I=\{1,i_0+2,\ldots,i_{d-1}+2\}$ for some $\{i_0,\ldots, i_{d-1}\}\in \mathbf{I}_{n+2d-2}^{d-1}$. 
\item The module $M_I$ is injective if and only if $I=\{i_0,\ldots,i_{d-1},n+2d\}$ for some $\{i_0,\ldots, i_{d-1}\}\in \mathbf{I}_{n+2d-2}^{d-1}$.
\item $\mathrm{Hom}_{A^d_n}(M_{i_0,\ldots,i_d},M_{j_0,\ldots,j_d})\ne 0\ \iff i_0-1< j_0< i_1-1<\cdots <i_d-1<j_d$. 
\end{itemize}
\end{theorem}}
\jjm{Note that, comparing with \cite[Corollary 3.7]{ot}, the set of indecomposable projective $(A^d_n)$-modules determines a maximal-by-size collection of non-intertwining $\nnw{(d+1)}$-subsets of $[n+2d]$. \nnw{Moreover, one can verify that this collection forms a slice, in the sense of Definition \ref{spread}.}}

\jjm{To describe the structure of tensor products of $d$-representation-finite algebras, we need the following notion.}
\begin{definition}\jdm{\cite[Definition 3.2]{is}}
For a finite-dimensional algebra $\nnw{\Lambda}$, a module $M\in\mathrm{mod}(\nnw{\Lambda})$ is \emph{$d$-precluster tilting} if it satisfies the following conditions:
\begin{itemize}
\item The module $M$ is a generator-cogenerator for $\mathrm{mod}(\nnw{\Lambda})$. 
\item \nw{We have $\tau_d(M) \in \mathrm{add}(M)$ and $\tau^-_d(M)\in \mathrm{add}(M)$.}
\item There is an equality $\mathrm{Ext}^i_{\nnw{\Lambda}}(M,M)=0$ for all $0<i<d$.
\end{itemize}
\end{definition}

Alternatively, $M$ is called \jjm{a} $d$-ortho-symmetric module $M$ in \cite{chenko}. For a $d$-precluster-tilting module $M$, the subcategory $\mathrm{add}(M)\subseteq \mathrm{mod}(\nnw{\Lambda})$ is called a $d$-precluster-tilting subcategory.
If $\Lambda$ is a finite-dimensional algebra of global dimension at most $d$, then the \emph{$(d+1)$-preprojective algebra} $\Pi_{\nnw{d+1}}(\Lambda)$ is defined to be the tensor algebra $\Pi_{\nnw{d+1}}(\Lambda) = T_\Lambda \mathrm{Ext}^{d}_{\Lambda}(D\Lambda,\Lambda)$. 

\jdm{
\begin{remark}
The $3$-preprojective algebras are closely related to the relation extensions of  Assem, Br\"ustle and Schiffler \cite{abs}.
For  tilted algebras, the two constructions can be identified (see \cite[Remark 3.24]{amiotsurvey}, \cite[Proposition 4.7]{amiot2}, \cite[Proposition 5.2.1]{amiot1}). This is helpful, as quivers for relation extensions can be explicitly described. In particular, Assem, Br\"ustle and Schiffler show that there are extra arrows coming from relations \cite[Lemma 2.3]{abs}. In the case of preprojective algebras, this was considered in terms of mesh relations in \cite{gls07} (see also \cite[Section 4.3]{amiot2} and \cite[Theorem 6.12]{keller08}). 
\end{remark}}
For a given $d$-representation-finite algebra $\Lambda$, then the preprojective algebra $\Pi_{d+1}\nnw{(\Lambda)}$ is self injective by \cite[Corollary 3.4]{iostable} and has a canonical $d$-precluster-tilting subcategory: the additive subcategory of projective-injective modules.  
\begin{lemma}\cite[Lemma 2.13]{iostable}
Let $\Lambda$ be a $d$-representation-finite algebra and $\Gamma=\nnw{\Pi_{d+1}(\Lambda)}$. Then $${_\Lambda}\Gamma \nw{\cong} \bigoplus_{i\geq 0}\tau_d^{-i}(\Lambda)$$ \nw{as $\Lambda$-modules.}
\end{lemma}

The subcategory  $\mathrm{add}\{\tau_d^{-i}(\Lambda)\mid i\geq 0\}\subseteq\mathrm{mod}(\Lambda)$ is the unique $d$-cluster-tilting subcategory of $\mathrm{mod}(\Lambda)$ by Theorem 1.6 of \cite{iy2}. Thus the $d$-precluster-tilting subcategory of $\mathrm{mod}(\Pi_{d+1}(\Lambda))$ consisting of projective-injective modules can be thought of as having as objects the $\tau^-_d$-orbits of all $\Lambda$-modules.
\begin{lemma}\label{welldefined} 
Let $A$ and $B$ be $d$-representation-finite algebras and $\Lambda=A^\mathrm{op}\otimes B$. Let $\Gamma=\nnw{\Pi_{2d+1}(\Lambda)}$. Then 
$${_\Lambda}\Gamma\nw{\cong}\bigoplus_{i\geq 0}\tau_d^{-i}(A^\mathrm{op})\otimes \tau_d^{-i}(B)$$ as $\Lambda$-modules.
\end{lemma}
\begin{proof}
Since $\mathrm{gl.dim}(A)\leq d$ and $\mathrm{gl.dim}(B)\leq d$, we have 
\begin{align*}\tau_d^-\otimes \tau_d^-&=\tau^-\Omega^{1-d}\otimes \tau^-\Omega^{1-d}\\
&=\mathrm{Ext}^1_{A^\mathrm{op}}(D(A^\mathrm{op}),\Omega^{1-d}-)\otimes \mathrm{Ext}^1_{B}(D(B),\Omega^{1-d}-)\\
&=\mathrm{Ext}^d_{A^\mathrm{op}}(D(A^\mathrm{op}),-)\otimes \mathrm{Ext}^d_B(DB,-)\\
&\cong\mathrm{Ext}^{2d}_\Lambda(D(\Lambda),-)\\
&=\mathrm{Ext}^{2d}_\Lambda(D(\Lambda),\Lambda)\otimes _\Lambda-.
\end{align*}
\end{proof}
Now we may prove the main result of this section.
\begin{theorem}\label{precl}
Let $A$ and $B$ be $d$-representation-finite algebras and $\Lambda=A^\mathrm{op}\otimes B$. Let $\Gamma=\nnw{\Pi_{2d+1}(\Lambda)}$. Then there exists a $d$-precluster-tilting subcategory $\mathcal{D}\subseteq\mathrm{mod}(\Gamma)$.
\end{theorem}

\begin{proof}
Let $\mathcal{C}\subseteq\mathrm{mod}(A^\mathrm{op})$ and $\mathcal{C}^\prime\subseteq \mathrm{mod}(B)$ be $d$-cluster-tilting subcategories.
Define the subcategory $\mathcal{D}\subseteq\mathrm{mod}(\Gamma)$ as having objects the $(\tau_d\otimes\tau_d)$-orbits of $\mathcal{C}\otimes \mathcal{C}^\prime$, which is well defined by Lemma \ref{welldefined}. 
By construction, the subcategory $\mathcal{D}$ contains all projective and injective $\Gamma$-modules, and is hence a generator-cogenerator for $\mathrm{mod}(\Gamma)$.  

\jm{By choosing a representative of its $(\tau_d\otimes \tau_d)$-orbit}, any object $X\in\mathcal{D}$ must be expressible as either $I\otimes N$ or $M\otimes J$, where $I,M\in\mathcal{C}$ \jm{and} $J,N\in\mathcal{C}^\prime$ and \jm{such that} $I$ (respectively $J$) is injective as an $A^\mathrm{op}$-module (respectively $B$-module.) So assume $X=M\otimes J$. Then the minimal injective resolution of $M$ as \jjm{a} $B$-module: 
$$0\rightarrow M\rightarrow I_0\rightarrow I_1\rightarrow \cdots \rightarrow I_d$$ induces a minimal injective resolution of $M\otimes J$ as \jjm{a} \nnw{$\Gamma$}-module:
$$0\rightarrow M\otimes J \rightarrow I_0\otimes J\rightarrow I_1\otimes J \rightarrow I_d\otimes J$$
and hence the functor $\tau_d^-$ in the category $\mathrm{mod}(\Gamma)$ sends the module $M\otimes J$ to $\tau_d^-(M)\otimes \nu^- (J)$, where each is considered as \nw{a representative} of its respective $(\tau_d\otimes\tau_d)$-orbit. Likewise $\tau_d^-(I\otimes N)=\nu^- (I)\otimes \tau_d^-(N)$. Hence $\tau_d^-(\mathcal{D})\subseteq \mathcal{D}$ and, dually,  $\tau_d(\mathcal{D})\subseteq \mathcal{D}$.

Finally, suppose there are $X,Y\in\mathcal{D}$ and an integer $0<i<d$ such that $$\mathrm{Ext}^i_{\Gamma}(X,Y)\ne0.$$ \jm{Let $X=M\otimes M^\prime$ and $Y=N \otimes N^\prime$ be representatives of their respective  $(\tau_d\otimes\tau_d)$-orbits}, there must be integers $i_A+i_B=i$, \jjm{where $0<i_A<d$ and $0<i_B<d$}, such that  $\mathrm{Ext}^{i_A}_{A^\mathrm{op}}(M,N)\ne0$ and $\mathrm{Ext}^{i_B}_{B}(M^\prime,N^\prime)\ne0$. This is a contradiction.  So $\mathrm{Ext}^i_{\Gamma}(\mathcal{D},\mathcal{D})=0$ for all $0<i<d$.
\end{proof}

\begin{example}\label{ex-precl}
Let $\Lambda=A_2^2\otimes A_2^2 $. Then $\Pi_{5}(\Lambda)$ has quiver (where we omit relations):
$$\begin{tikzpicture} [scale=1.5]
\draw [->] (4.,0.2) -- (4.,1.8);
\draw [->] (4.,2.2) -- (4.,3.8);
\draw [->] (4.3,0.) -- (5.5,0.);
\draw [->] (6.,0.2) -- (6.,1.8);
\draw [->] (4.3,2.) -- (5.5,2.);
\draw [->] (6.,2.2) -- (6.,3.8);
\draw [->] (4.3,4.) -- (5.5,4.);
\draw [->] (6.5,0.) -- (7.7,0.);
\draw [->] (8.,0.2) -- (8.,1.8);
\draw [->] (6.5,2.) -- (7.7,2.);
\draw [->] (8.,2.2) -- (8.,3.8);
\draw [->] (6.5,4.) -- (7.7,4.);
\draw [->] (7.8,3.8) to [bend left= 25] (4.2,0.2);

\node(A) at (3.75,-0) {\small{$13\otimes 13$}};
\node(B) at (3.75,2){\small{$14\otimes 13$}};
\node(C) at (3.75,4) {\small{$15\otimes 13$}};
\node(D) at(6,-0) {\small{$13\otimes 14$}};
\node(E) at(6,2){\small{$14\otimes 14$}};
\node(F) at(6,4.) {\small{$14\otimes 13$}};
\node(G) at(8.25,0.) {\small{$13\otimes 15$}};
\node(H) at(8.25,2.) {\small{$14\otimes 15$}};
\node(I) at(8.25,4.) {\small{$15\otimes 15$}};

\end{tikzpicture}$$
The $2$-precluster-tilting subcategory $\mathcal{C}\subseteq \mathrm{mod}(\Pi_{5}(\Lambda))$ can be \nnw{described as follows}.

$$\begin{tikzpicture}[x=1.75cm,y=1.0cm,scale=1.5]
\draw [->] (2+4.,-6 .3+0.6+2) -- (2+4.,-5 .3+2);
\draw [->] (2+4.,-5 .3+0.6+2) -- (2+4.,-4 .3+2);
\draw [->] (2+4.,-4 .3+0.6+2) -- (2+4.,-3 .3+2);
\draw [->] (2+5.,-6.3+0.6+2) -- (2+5.,-5.3+2);
\draw [->] (2+5.,-5.3+0.6+2) -- (2+5.,-4.3+2);
\draw [->] (2+5.,-4.3+0.6+2) -- (2+5.,-3.3+2);
\draw [->] (2+6.,-6.3+0.6+2) -- (2+6.,-5.3+2);
\draw [->] (2+6.,-5.3+0.6+2) -- (2+6.,-4.3+2);
\draw [->] (2+6.,-4.3+0.6+2) -- (2+6.,-3.3+2);
\draw [->] (2+7.,-6.3+0.6+2) -- (2+7.,-5.3+2);
\draw [->] (2+7.,-5.3+0.6+2) -- (2+7.,-4.3+2);
\draw [->] (2+7.,-4.3+0.6+2) -- (2+7.,-3.3+2);
\draw [->] (2+7.,-3.3+0.6+2) -- (2+7.,-2.3+2);
\draw [->] (2+7.,-2.3+0.6+2) -- (2+7.,-1.3+2);
\draw [->] (2+8.,-3.3+0.6+2) -- (2+8.,-2.3+2);
\draw [->] (2+8.,-2.3+0.6+2) -- (2+8.,-1.3+2);
\draw [->] (2+9.,-3.3+0.6+2) -- (2+9.,-2.3+2);
\draw [->] (2+9.,-2.3+0.6+2) -- (2+9.,-1.3+2);

\draw [->] (2+4.4,-6.+2) -- (2+5.3-0.7,-6.+2);
\draw [->] (2+5.4,-6.+2) -- (2+6.3-0.7,-6.+2);
\draw [->] (2+6.4,-6.+2) -- (2+7.3-0.7,-6.+2);
\draw [->] (2+4.4,-5.+2) -- (2+5.3-0.7,-5.+2);
\draw [->] (2+5.4,-5.+2) -- (2+6.3-0.7,-5.+2);
\draw [->] (2+6.4,-5.+2) -- (2+7.3-0.7,-5.+2);
\draw [->] (2+4.4,-4.+2) -- (2+5.3-0.7,-4.+2);
\draw [->] (2+6.4,-4.+2) -- (2+7.3-0.7,-4.+2);
\draw [->] (2+5.4,-4.+2) -- (2+6.3-0.7,-4.+2);
\draw [->] (2+4.4,-3.+2) -- (2+5.3-0.7,-3.+2);
\draw [->] (2+5.4,-3.+2) -- (2+6.3-0.7,-3.+2);
\draw [->] (2+6.4,-3.+2) -- (2+7.3-0.7,-3.+2);
\draw [->] (2+7.4,-3.+2) -- (2+8.3-0.7,-3.+2);
\draw [->] (2+8.4,-3.+2) -- (2+9.3-0.7,-3.+2);
\draw [->] (2+7.4,-2.+2) -- (2+8.3-0.7,-2.+2);
\draw [->] (2+8.4,-2.+2) -- (2+9.3-0.7,-2.+2);
\draw [->] (2+7.4,-1.+2) -- (2+8.3-0.7,-1.+2);
\draw [->] (2+8.4,-1.+2) -- (2+9.3-0.7,-1.+2);

\begin{scriptsize}{\tiny}

\node(2+A+2) at (2+4.,-6.+2) {$M_{135}\otimes M_{135}$};
\node(2+B+2) at (2+4.,-5.+2) {$M_{135}\otimes M_{136}$};
\node(2+C+2) at (2+4.,-4.+2) {$M_{135}\otimes M_{146}$};
\node(2+D+2) at (2+4.,-3.+2) {$M_{135}\otimes M_{246}$};
\node(2+E+2) at (2+5.,-6.+2) {$M_{136}\otimes M_{135}$};
\node(2+F+2) at (2+5.,-5.+2) {$M_{136}\otimes M_{136}$};
\node(2+G+2) at (2+5.,-4.+2) {$M_{136}\otimes M_{246}$};
\node(2+H+2) at (2+5.,-3.+2) {$M_{136}\otimes M_{246}$};
\node(2+I+2) at (2+6.,-6.+2) {$M_{146}\otimes M_{135}$};
\node(2+J+2) at (2+6.,-5.+2) {$M_{146}\otimes M_{136}$};
\node(2+K+2) at (2+6.,-4.+2) {$M_{146}\otimes M_{146}$};
\node(2+L+2) at (2+6.,-3.+2) {$M_{146}\otimes M_{246}$};
\node(2+m+2) at (2+7.,-6.+2) {$M_{246}\otimes M_{135}$};
\node(2+n+2) at (2+7.,-5.+2) {$M_{246}\otimes M_{136}$};
\node(2+o+2) at (2+7.,-4.+2) {$M_{246}\otimes M_{146}$};
\node(2+p+2) at (2+7.,-3.+2) {$M_{246}\otimes M_{246}$};
\node(2+q+2) at (2+7.,-2.+2) {$M_{135}\otimes M_{136}$};
\node(2+r+2) at (2+7.,-1.+2) {$M_{135}\otimes M_{146}$};
\node(2+s+2) at (2+8.,-3.+2) {$M_{136}\otimes M_{135}$};
\node(2+t+2) at (2+8.,-2.+2) {$M_{136}\otimes M_{136}$};
\node(2+u+2) at (2+8.,-1.+2) {$M_{136}\otimes M_{146}$};
\node(2+v+2) at (2+9.,-3.+2) {$M_{146}\otimes M_{135}$};
\node(2+w+2) at (2+9.,-2.+2) {$M_{146}\otimes M_{136}$};
\node(2+x+2) at (2+9.,-1.+2) {$M_{146}\otimes M_{146}$};
\end{scriptsize}

\end{tikzpicture}$$
\end{example}

\section{Mutations}

Mutations of non-$l$-intertwining collections are a generalisation of both mutations of cyclic polytope triangulations and mutations of non-crossing collections. It is unfortunately already true in both classical cases that mutation is not always possible. The inherited lack of mutatability makes a criterion for mutatability for non-$l$-intertwining collections difficult. In this section we rather outline how mutations are expected to work.

\subsection{Higher APR tilting}
In this section we establish combinatorics on non-$l$-intertwining collections that are compatible with mutation of \nw{tilting modules in} \nnw{$d$}-precluster-tilting subcategories.
\jjm{
\begin{definition}\label{tiltdef}
Let $\Lambda$ be a finite-dimensional algebra.
An $\Lambda$-module $T$ is a \emph{$d$-tilting module} \cite{ha,miya} if:
\begin{enumerate}
\item $\mathrm{proj.dim}(T)\leq d$.\label{tilt1}
\item $\mathrm{Ext}^i_\Lambda(T,T)=0$ for all $0<i\leq d$.\label{tilt2}
\item if there exists an exact sequence $$0\rightarrow \Lambda\rightarrow T_0\rightarrow T_1\rightarrow \cdots \rightarrow T_d\rightarrow 0$$ 
where $T_0,\ldots, T_d\in \mathrm{add}(T)$.\label{tilt3}
\end{enumerate}
\end{definition}}
Since non-$l$-intertwining collections are generalisations of clusters in the cluster structure Grassmannian coordinate ring, we wish to have a notion of mutation between them. 

Let $\Lambda$ be a $d$-representation-finite algebra and $\mathcal{C}\subseteq \mathrm{mod}(\Lambda)$ a $d$-cluster-tilting subcategory. Let $i$ be a sink in the quiver of $\Lambda$, let $P_i$ be the corresponding simple projective $\Lambda$-module and decompose $\Lambda=P_i\oplus Q$. Then
$$T:=(\tau_d^-P_i)\oplus Q$$ is the \emph{$d$-APR tilting module} associated with $P_i$, and $\mathrm{End}_\Lambda(T)^\mathrm{op}$ is the \emph{$d$-APR tilt of $\Lambda$} associated with $P_i$ \cite[Definition 3.1]{io}. 
Higher representation finiteness is preserved by higher APR tilts, as the following result shows.

\begin{theorem}\cite[Theorem 4.2]{io}\label{dAPR}
Let $\Lambda$ be a $d$-representation-finite algebra and let \jm{$\Lambda^\prime$} be a $d$-APR tilt of $\Lambda$. Then \jm{$\Lambda^\prime$} is $d$-representation finite.
\end{theorem}

\begin{definition}
Let $A$ and $B$ be $d$-representation-finite algebras and $\Lambda=A^\mathrm{op}\otimes B$, with $\Gamma=\nnw{\Pi_{2d+1}(\Lambda)}$. \jm{We consider each $\Lambda$-module $M\in\mathcal{C}\otimes \mathcal{C^\prime}$ as a representative of its $(\tau_{d} \otimes \tau_{d})$-orbit in $\mathrm{mod}(\Gamma)$ as in Theorem \ref{precl}.}
Given a sink $i$ in $A$ and sink $j$ in $B$, there is \jm{a} simple projective $B$-module $Be_j$. Let $P=e_iA\otimes Be_j$ and decompose $\Lambda=P\oplus Q$.
The module $$\jm{_\Gamma}T:=(e_iA\otimes \tau_d^-(Be_j))\oplus Q$$ is the  \emph{$d$-APR cotilting-tilting module} \jm{associated with $P$}.
Dually, there is corresponding simple \nw{injective} $\nw{A^{\mathrm{op}}}$-module $\nw{D(Ae_i)}$. Let $I=D(Ae_i)\otimes D(e_jB)$ and decompose $\nw{D\Lambda}=I\oplus J$.
The module $$\nw{_\Gamma}C:=(\nw{\tau_d(D(Ae_i))\otimes D(e_jB))\oplus J}$$ is the \emph{$d$-APR tilting-cotilting module} \jm{associated with $I$}.
\end{definition} 

\jjm{\begin{lemma}\label{onapr}
Let $A$ and $B$ be $d$-representation-finite algebras and $\Lambda=A^\mathrm{op}\otimes B$, with $\Gamma=\Pi_{2d+1}(\Lambda)$.
Let $T$ be the $d$-APR cotilting-tilting $\Gamma$ module associated with $e_iA\otimes Be_j$. Then $T$ is a $d$-tilting $\Gamma$-module.
\end{lemma}
\begin{proof} Let $B= Be_{j}\oplus B'$ as a $B$-module. Since $B(1-e_j)\oplus \tau_d^-(Be_j)$ is $d$-APR tilting, there is an exact sequence in $\mathrm{mod}(B)$:
$$0 \rightarrow Be_{j} \rightarrow T_{1} \rightarrow \dots \rightarrow T_{d} \rightarrow \tau_{d}^-(Be_{j}) \rightarrow 0.$$
such that $T_i\in\mathrm{add}(B(1-e_j)\oplus \tau_d^-(Be_j)\nnw{)}\cap\mathrm{add}(B)$ for all $1\leq i\leq d$.
This induces an exact sequence in $\mathrm{mod}(\Gamma)$:
$$0 \rightarrow e_{i}A \otimes Be_{j} \rightarrow e_{i}A \otimes T_{1} \rightarrow \dots \rightarrow e_{i}A \otimes T_{d} \rightarrow e_{i}A \otimes \tau_{d}^-(Be_{j}) \rightarrow 0$$
 such that $e_iA\otimes T_i\in\mathrm{add}(T)\cap\mathrm{add}(\Gamma)$ for all $1\leq i\leq d$. It then follows that $_{\Gamma}T$ satisfies axioms (\ref{tilt1}) and (\ref{tilt3}) of Definition \ref{tiltdef}. Finally, \begin{align*}
 \mathrm{Ext}^d_\Gamma(T,T)&\cong \mathrm{Ext}^d_\Gamma(e_iA\otimes \tau_d^-(Be_j),T)=0,\end{align*}
since \begin{align*}
\mathrm{Ext}^d_B(\tau_d^-(Be_j),B(1-e_j)\oplus \tau_d^-(Be_j))=0.
 \end{align*}
\end{proof}
}
Denote by $D^{b}(\nnw{\Lambda})$ the bounded derived category of $\mathrm{mod}(\nnw{\Lambda})$, and by $$\nu:= D\nnw{\Lambda}\otimes_{\nnw{\Lambda}}^\mathbf{L}-\cong D\mathbf{R}\mathrm{Hom}_{\nnw{\Lambda}}(-,\nnw{\Lambda}):D^b(\nnw{\Lambda})\rightarrow D^b(\nnw{\Lambda})$$
the Nakayama functor in $D^b(\nnw{\Lambda})$. Further denote by $\nu_d:=\nu[-d]$ the $d$-th desuspension of $\nu$. Note that $\tau_d=H_0(\nu_d)$ (see for example Section 2.2 of \cite{io}).
\begin{lemma}\label{boring}
Let $i$ be a sink in $A$ and $j$ a sink in $B$. Let $T$ be the \nw{corresponding} $d$-APR cotilting-tilting $\Gamma$-module and let $C$ be the \nw{corresponding} $d$-APR tilting-cotilting $\nw{\Gamma}$-module. Then, considering $T$ \nw{and $C$ as complexes in $D^b(\Gamma)$} we have an isomorphism
$$\mathrm{End}_{D^b(\Gamma)}(T)^\mathrm{op}\cong \mathrm{End}_{D^b(\nw{\Gamma})}(C)^{\nw{\mathrm{op}}}.$$
\end{lemma}
\begin{proof}
As a complex we have $T=(e_iA\otimes \nu_d^-(Be_j))\oplus Q$. Likewise, as a complex we have $C:=(\nu_d(\nw{D}(Ae_i))\otimes \nw{D}(e_jB))\oplus J$. 
Define  $$T^\prime:=(\nu_d(e_iA)\otimes Be_j)\oplus Q.$$
Then $\nu(T')=C$ \jjm{and} $$\mathrm{End}_{D^b(\nw{\Gamma})}(C)^{\nw{\mathrm{op}}}\cong \mathrm{End}_{D^b(\Gamma)}(T^\prime)^\mathrm{op},$$ \nw{since $\nu$ is an auto-equivalence.} Since $T^\prime$ is in the same $\jjm{(}\nu_d\otimes \nu_d\jjm{)}$-orbit as $T$, we further have
$$\mathrm{End}_{D^b(\Gamma)}(T^\prime)^\mathrm{op}\cong \mathrm{End}_{D^b(\Gamma)}(T)^\mathrm{op}.$$ This completes the proof.
\end{proof}

Let $A$ and $B$ be $d$-representation-finite algebras and $\Lambda=A^\mathrm{op}\otimes B$, with $\Gamma=\Pi_{2d+1}(\Lambda)$. For a $d$-APR cotilting-tilting $\Gamma$-module $T$, we do not presently show whether there is a $d$-precluster-tilting subcategory 
$$\mathcal{D}^\prime \subseteq\mathrm{mod}(\mathrm{End}_{D^b(\Gamma)}(T)^\mathrm{op}).$$ One might suspect that similar methods to those used in \cite[Section 5]{ot} apply. One hindrance is determining whether $(d+2)$-angulated categories can be constructed.

\subsection{Uniterated mutation of non-$l$-intertwining collections}

Triangulations of cyclic polytopes can be mutated. In terms of the associated non-intertwining collections, two such collections are said to be a mutation of one other if they differ by a single element.
 For convenience, we write $I+1$ for $\{ i+1\nw{\pmod{n}} \mid i \in I \}$ and use $I-1$ similarly. 

\begin{definition}\label{def-mut}
Given a non-$l$-intertwining collection $\mathcal{C} \subseteq \subs{n}{k}{ l}$ with $|\mathcal{C}| = \binom{k-1}{l-1}\binom{\nw{n-k-1}}{l-1}$ and $I \in \mathcal{C}$, we define 
$$\mu_{I}^{+}(\mathcal{C}) := \mathcal{C} \setminus \{ I \} \cup \{ I+1 \},$$
$$\mu_{I}^{-}(\mathcal{C}) = \mathcal{C} \setminus \{ I \} \cup \{ I-1 \},$$
\nw{provided that these collections are non-$l$-intertwining.}
\end{definition}

\begin{example}\label{ex-bad-mut}
We continue from Examples \ref{ex-spread} and \ref{ex-construction}. We have $\mathcal{C} = \mathcal{C}_{4}(\mathcal{T}_1) = \{ 1356, 1346, 1367, 1347, 1246, 1467, 1247, 1457, 1257 \}$. Then:
\begin{align*}
\mu_{1356}^{+}(\mathcal{C}) = \{ 2467, 1346, 1367, 1246, 1467, 1247, 1457, 1257 \}
\end{align*}
We may also mutate $\mathcal{C}$ in a different way to Definition \ref{def-mut} by replacing the $4$-tuple $\nnw{1246}$ with the 4-tuple $\nnw{1357}$ to obtain a non-$l$-intertwining collection. However this case lies beyond the scope of our paper, since $\nnw{1357}$ is not contained in  $\subs{\nw{8}}{4}{3}$ \nnw{and the mutation is not of the form in Definition \ref{def-mut}.}
\end{example}




\begin{proposition}\label{prop-mut}
Let $\mathcal{T}$ be a non-intertwining collection of $\binom{\nw{n-l-1}}{l-1}$ elements in $\nonconsec{n}{l-1}$ that is a slice and $l \leq k \leq \nnw{n-l}$. Let $T \in \mathcal{T}$ such that $\mu_{T}^{+}(\mathcal{T})$ is a \nnw{slice contained in} \jjm{$\nonconsec{n}{l-1}$}. Then there is a sequence of mutations from $\mathcal{C}_{k}(\mathcal{T})$ to $\mathcal{C}_{k}(\mu_{T}^+\mathcal{T})$.
\end{proposition}
\begin{proof}
Let $T = \{ t_{1}, \dots ,t_{l} \}$. By Lemma \ref{int-endpoints-lemma}, two elements $I\in\mathcal{C}_k(\mathcal{T})$ and $J\in\mathcal{C}_k(\mu_{T}^{+}\mathcal{T})$ are $l$-intertwining if and only if $\widehat{J}=T=\widehat{I}+1$. 
But then one must have
\begin{align*}
I &= [t_{1},t_{1}+a_{1}] \dots [t_{l},t_{l}+a_{l}] \\
J &= [t_{1}+1,t_{1}+1+b_{1}] \dots [t_{l}+1,t_{l}+1+b_{l}]
\end{align*}
for some $a_{i}, b_{i} \geq 0$ such that $\sum a_{i} = \sum b_{i} = k-l$. If $b_{j} < a_{j}$ for some $j$, then the $j$-th interval of $J$ contains the $j$-th interval of $I_{s}$, and so the two cannot be $l$-intertwining. Hence $J=I+1$. This means that one can mutate from $\mathcal{C}_{k}(\mathcal{T})$ to $\mathcal{C}_{k}(\mu_{T}^{+}\mathcal{T})$ using mutations \jjm{one after the other} at those $I \in C_{k}(\mathcal{T})$ such that $\widehat{I}=T$. 
\end{proof}

Now we study how mutation compares with \nnw{$d$}-precluster-tilting subcategories.
 Given an $l$-ple interval
$$I =[t_{1}, t'_{1}] \sqcup \dots \sqcup [t_{i}, t'_{i}] \sqcup \dots \sqcup [t_{l}, t'_{l}]\jjm{,} $$ \jjm{w}e define 
$$v_{i}(I):= [t_{1}, t'_{1}] \sqcup \dots \sqcup [t_{i-1}, t'_{i-1}] \sqcup [t_{i}+1, t'_{i}+1] \sqcup [t_{i+1}, t'_{i+1}] \sqcup \dots \sqcup [t_{l}, t'_{l}].$$
Likewise we define 
$$h_{i}(I): = [t_{1}, t'_{1}] \sqcup \dots \sqcup [t_{i-1}, t'_{i-1}-1] \sqcup [t_{i}-1, t'_{i}] \sqcup \dots \sqcup [t_{l}, t'_{l}]\jjm{.}$$
Note that $h_{i}$ is the inverse of the operation $v_{i}$ on the complement of $I$. \nw{Note further that if $t'_{i}=t_{i+1}-1$ then applying $v_{i}$ no longer gives an $l$-ple interval. Similarly, if $t_{i-1}=t'_{i-1}$, then applying $h_{i}$ does not give an $l$-ple interval.}

\begin{proposition}\label{quiv-index}
Let $\mathcal{T}$ be the non-intertwining collection of $l$-subsets of $[n]$ labelling the \jjm{indecomposable projective $(A^{l-1}_{n-2l+1})$-modules} for the higher Auslander algebra $A^{l-1}_{\nw{n-2l+1}}$. Then
\begin{itemize}
\item  The vertices of the algebra $(A^{l-1}_{k-l+1})^{\mathrm{op}}\otimes A^{l-1}_{\nw{n-k-l+1}}$ are indexed by $\mathcal{C}_k(\mathcal{T})$. 
\item The arrows of the quiver of $(A^{l-1}_{k-l+1})^{\mathrm{op}}\otimes A^{l-1}_{\nw{n-k-l+1}}$ are given by $v$ and $h$.
\end{itemize} 
\end{proposition}
\begin{proof}
For $k=l$ this is precisely Construction \ref{ARdefn}. 
As in the  proof of Theorem \ref{maintheorem}, there is a bijection between $\mathcal{C}_k(\mathcal{T})$  and the set \[\mathcal{S}:= \{ ((q_{1}, \dots , q_{l}),(r_{1}, \dots , r_{l})) \in \mathbb{Z}_{\geq 0}^{l}\times \mathbb{Z}_{\geq 0}^{l} \mid \sum_{i=1}^{l}q_{i}=k-l, \sum_{i=1}^{l}r_{i}=\nw{n-k-l} \}.\]  
\nw{As in the proof of Theorem \ref{maintheorem}} this means the vertices of the algebra $(A^{l-1}_{k-l+1})^{\mathrm{op}}\otimes A^{l-1}_{\nw{n-k-l+1}}$ are indexed by $\mathcal{C}_{k}(\mathcal{T})$.  
By Construction \ref{ARdefn}, for each given $\underline{q}$, the arrows from $(\underline{q},\underline{r}_1)$ to $(\underline{q},\underline{r}_2)$ are determined by $v$. However $h$ is simply inverse of $v$ on the complements, and hence describes the arrows from $(\underline{q}_1,\underline{r})$ to $(\underline{q}_2,\underline{r})$ for each $\underline{r}$.
\end{proof}

Now let $\Lambda_{n,k}^{l-1}:= (A^{l-1}_{k-l+1})^{\mathrm{op}}\otimes A^{l-1}_{n-k-l+1}$ and $\Gamma_{n,k}^{l-1}:=\Pi_{2d+1}(\Lambda_{n,k}^{l-1})$. Let $\mathcal{D}_{n,k}^{l-1} \subseteq \mathrm{mod}(\Gamma_{n,k}^{l-1})$ be the $(l-1)$-precluster-tilting subcategory from Theorem \ref{precl}. We denote the indecomposables of this category by $\mathrm{ind}\jjm{(}\mathcal{D}_{n,k}^{l-1}\jjm{)}$.

\begin{lemma}\label{cat-desc}
\jjm{Each element $M\in\mathrm{ind}(\mathcal{D}_{n,k}^{l-1})$ has a label $\iota(M)\in\subs{n}{k}{l}$, such that irreducible morphisms in $\mathcal{D}_{n,k}^{l-1}$ are determined by $h$ and $v$. }
\end{lemma}
\begin{proof}
\jjm{For the projectives in $\mathcal{D}_{n,k}^{l-1}$, such a label is found by using Proposition \ref{quiv-index}.}
The remaining indecomposables of $\mathcal{D}_{n,k}^{l-1}$ are reachable by repeated application of $\tau_{l-1}^{-} \otimes 1$ and $1 \otimes\tau_{l-1}^{-}$. Following \cite[Proposition 3.13]{ot}, we define $\iota((\tau_{\jjm{l-1}}^{-}\otimes 1)M)=\iota(M)-1$ and $\iota((1\otimes \tau_{l-1}^{-})M)=\iota(M)+1$. This \jjm{is well defined}, since $\iota(M)=\iota(N)$ if and only if $M$ and $N$ are in the same $(\tau_{l-1}\otimes\tau_{l-1})$-orbit.
\jdm{Recall from Theorem \ref{oppthom2} that
$$\mathrm{Hom}_{A^d_n}(M_{i_0,\ldots,i_d},M_{j_0,\ldots,j_d})\ne 0\ \iff i_0-1< j_0< i_1-1<\cdots< i_d-1<j_d.$$} In particular, $\iota$ is automatically compatible with irreducible morphisms.
\end{proof}

As in Construction \ref{ARdefn}, the quiver of $A^{l-1}_{\nw{n-2l+1}}$ has a unique source \jjm{given by the vertex $\nnw{(1,3,\ldots,2l-3)}.$} \jjm{By Theorem \ref{oppthom2},} this determines a simple projective $(A^{l-1}_{\nw{n-2l+1}})^\mathrm{op}$-module \jjm{$M_{1,3,\ldots, \nnw{2l-1}}$.} \nnw{Let $J=(1,3,\dots, 2l-1)$.} Let $\mathcal{T} \subseteq \nonconsec{n}{l}$ be the slice labelling the \nnw{projectives} of $\nnw{(}A^{l-1}_{n-2l+1}\nnw{)^{\mathrm{op}}}$. It follows that there is a unique $I \in \mathcal{C}_{k}(\mathcal{T}$)  such that $\nw{\widehat{I}=J}$, given by $$\nw{I:=(1,3,\ldots,2l-1,2l,\ldots, l+k-1).}$$

\begin{proposition}\label{quiv-label}
Let $S_J$ be \nnw{the} simple projective $(A^{l-1}_{\nw{n-2l+1}})^\mathrm{op}$-module \nnw{contained in the $d$-cluster-tilting subcategory of $\mathrm{mod}(A_{n-2l+1}^{l-1})^{\mathrm{op}}$} \jjm{and let $I$ be the unique subset such that $\widehat{I}=J$.} Then \jjm{there is} a $(l-1)$-APR cotilting-tilting module $T$ such that \nw{the vertices of the algebra $\mathrm{End}_{(A^{l-1}_{k-l+1})^{\mathrm{op}}\otimes A^{l-1}_{\nw{n-k-l+1}}}(T)^\mathrm{op}$ are indexed by $\mathcal{C}_k(\mu_I^+(\mathcal{T}))$.} 
\end{proposition}
\begin{proof}
 Observe that \nw{$J$} corresponds to the element $\nw{(0,\ldots,0,n-2l+1)}\in\mathbb{Z}_{\geq 0}^{l}$ \jm{in the image of the map $\phi$}. Hence under the bijection in Proposition \ref{quiv-index}, $I$ labels the vertex $((0,\dots, 0, k-l),(0, \dots, 0, n-k-l+1))$ of $(A^{l-1}_{k-l+1})^{\mathrm{op}}\otimes A^{l-1}_{n-k-l+2}$. So $I$ corresponds to a source \jm{$i$} of $(A^{l-1}_{k-l+1})^{\mathrm{op}}$ and a sink \jm{$j$} of $A^{l-1}_{\nw{n-k-l+1}}$. Define the projective $((A^{l-1}_{k-l+1})^{\mathrm{op}}\otimes A^{l-1}_{\nw{n-k-l+1}})$-module $P_I=(\jm{e_{i}A^{l-1}_{k-l+1}}\otimes A^{l-1}_{\nw{n-k-l+1}}e_{j})$ \jm{and let $T$ be the $(l-1)$-APR cotilting-tilting module associated with $P_I$}. Then we have $\iota((1 \otimes \tau_{l-1}^{-})P_{I})=I+1\ \jjm{\in \mathcal{C}_k(\mu_I^+(\mathcal{T}))}$, so the result follows.
\end{proof}

\begin{example}
We illustrate this theorem with an example.
Let $l=3, n=\nw{8}, k=4$. \nw{The quiver of $A_{n-2l+1}^{l-1}=A_{3}^{2}$ is labelled as follows.}
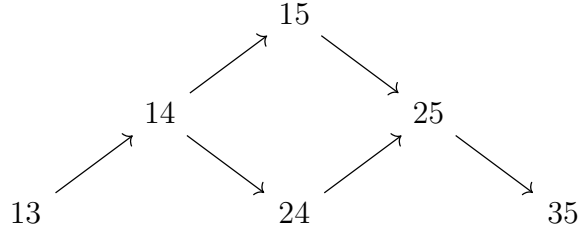
\[
\begin{tikzcd}
\ && 15 \arrow{dr} && \\
\ & 14 \arrow{ur} \arrow{dr} && 25 \arrow{dr} & \\
13 \arrow{ur} && 24 \arrow{ur} && 35
\end{tikzcd}
\]
The projective $(A_{3}^{2})^\mathrm{op}$-module \jjm{$e_{13}(A_{3}^{2})$} is simple. \jjm{The indecomposable projective $A$-modules $$M_{135},M_{136},M_{137},M_{146},M_{147},M_{157}$$ determine a non-intertwining collection \nnw{$\mathcal{T}$} of six $3$-subsets of $[7]$:
$$135,136,136,146,147,157.$$} Let $\Lambda=(A_{2}^{2})^{\mathrm{op}} \otimes A_{2}^{2}$ and let $\Gamma=\Pi_{5}(\Lambda)$. The algebra $\Lambda$ has the following quiver, which is labelled by $\mathcal{C}_{4}(\mathcal{T})$ as in Proposition \jjm{\ref{quiv-index}}.
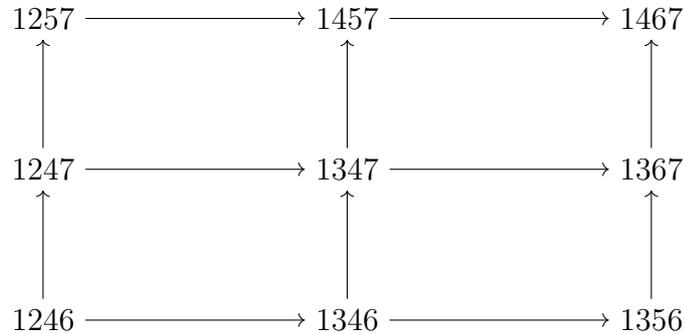
\[
\begin{tikzpicture}
\node(00) at (0,0){1246};
\node(20) at (4,0){1346};
\node(40) at (8,0){1356};
\node(02) at (0,2){1247};
\node(22) at (4,2){1347};
\node(42) at (8,2){1367};
\node(04) at (0,4){1257};
\node(24) at (4,4){1457};
\node(44) at (8,4){1467};

\draw (00) edge[->] (02);
\draw (02) edge[->] (04);
\draw (20) edge[->] (22);
\draw (22) edge[->] (24);
\draw (40) edge[->] (42);
\draw (42) edge[->] (44);

\draw (00) edge[->] (20);
\draw (20) edge[->] (40);
\draw (02) edge[->] (22);
\draw (22) edge[->] (42);
\draw (04) edge[->] (24);
\draw (24) edge[->] (44);
\end{tikzpicture}
\]
The algebra $\Gamma$ has the following quiver by \jjm{Lemma} \ref{cat-desc}. 
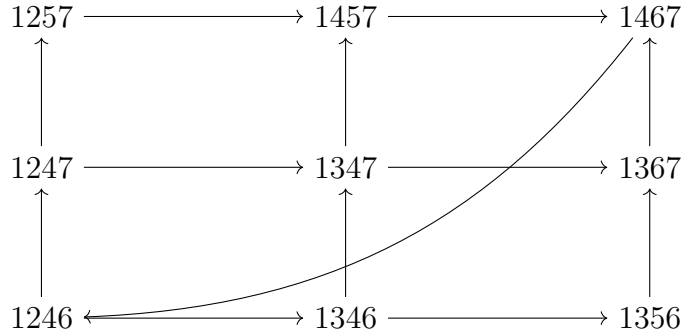
\[
\begin{tikzpicture}
\node(00) at (0,0){1246};
\node(20) at (4,0){1346};
\node(40) at (8,0){1356};
\node(02) at (0,2){1247};
\node(22) at (4,2){1347};
\node(42) at (8,2){1367};
\node(04) at (0,4){1257};
\node(24) at (4,4){1457};
\node(44) at (8,4){1467};

\draw (00) edge[->] (02);
\draw (02) edge[->] (04);
\draw (20) edge[->] (22);
\draw (22) edge[->] (24);
\draw (40) edge[->] (42);
\draw (42) edge[->] (44);

\draw (00) edge[->] (20);
\draw (20) edge[->] (40);
\draw (02) edge[->] (22);
\draw (22) edge[->] (42);
\draw (04) edge[->] (24);
\draw (24) edge[->] (44);

\draw[->] (44) to [bend left =25] (00);
\end{tikzpicture}
\]
As in Example \ref{ex-precl}, the 2-precluster-tilting subcategory of $\Gamma$ is shown below. It is now labelled according to \jjm{Lemma} \ref{cat-desc}. 
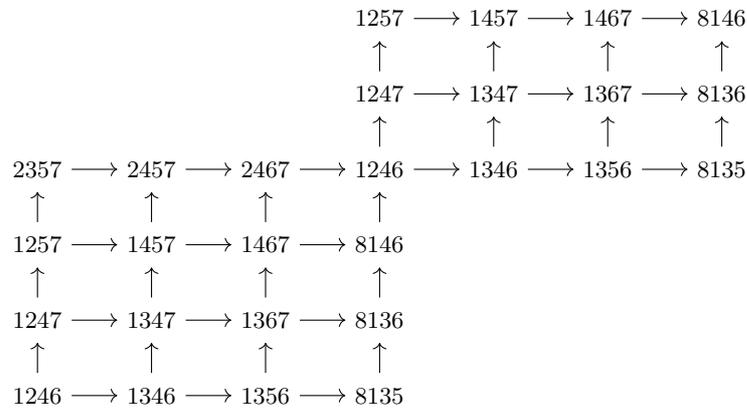
$$\begin{tikzpicture}[x=1.5cm,y=1.0cm,scale=1]
\draw [->] (2+4.,-6 .3+0.6+2) -- (2+4.,-5 .3+2);
\draw [->] (2+4.,-5 .3+0.6+2) -- (2+4.,-4 .3+2);
\draw [->] (2+4.,-4 .3+0.6+2) -- (2+4.,-3 .3+2);
\draw [->] (2+5.,-6.3+0.6+2) -- (2+5.,-5.3+2);
\draw [->] (2+5.,-5.3+0.6+2) -- (2+5.,-4.3+2);
\draw [->] (2+5.,-4.3+0.6+2) -- (2+5.,-3.3+2);
\draw [->] (2+6.,-6.3+0.6+2) -- (2+6.,-5.3+2);
\draw [->] (2+6.,-5.3+0.6+2) -- (2+6.,-4.3+2);
\draw [->] (2+6.,-4.3+0.6+2) -- (2+6.,-3.3+2);
\draw [->] (2+7.,-6.3+0.6+2) -- (2+7.,-5.3+2);
\draw [->] (2+7.,-5.3+0.6+2) -- (2+7.,-4.3+2);
\draw [->] (2+7.,-4.3+0.6+2) -- (2+7.,-3.3+2);
\draw [->] (2+7.,-3.3+0.6+2) -- (2+7.,-2.3+2);
\draw [->] (2+7.,-2.3+0.6+2) -- (2+7.,-1.3+2);
\draw [->] (2+8.,-3.3+0.6+2) -- (2+8.,-2.3+2);
\draw [->] (2+8.,-2.3+0.6+2) -- (2+8.,-1.3+2);
\draw [->] (2+9.,-3.3+0.6+2) -- (2+9.,-2.3+2);
\draw [->] (2+9.,-2.3+0.6+2) -- (2+9.,-1.3+2);
\draw [->] (2+10.,-3.3+0.6+2) -- (2+10.,-2.3+2);
\draw [->] (2+10.,-2.3+0.6+2) -- (2+10.,-1.3+2);

\draw [->] (2+4.3,-6.+2) -- (2+5.3-0.6,-6.+2);
\draw [->] (2+5.3,-6.+2) -- (2+6.3-0.6,-6.+2);
\draw [->] (2+6.3,-6.+2) -- (2+7.3-0.6,-6.+2);
\draw [->] (2+4.3,-5.+2) -- (2+5.3-0.6,-5.+2);
\draw [->] (2+5.3,-5.+2) -- (2+6.3-0.6,-5.+2);
\draw [->] (2+6.3,-5.+2) -- (2+7.3-0.6,-5.+2);
\draw [->] (2+4.3,-4.+2) -- (2+5.3-0.6,-4.+2);
\draw [->] (2+6.3,-4.+2) -- (2+7.3-0.6,-4.+2);
\draw [->] (2+5.3,-4.+2) -- (2+6.3-0.6,-4.+2);
\draw [->] (2+4.3,-3.+2) -- (2+5.3-0.6,-3.+2);
\draw [->] (2+5.3,-3.+2) -- (2+6.3-0.6,-3.+2);
\draw [->] (2+6.3,-3.+2) -- (2+7.3-0.6,-3.+2);
\draw [->] (2+7.3,-3.+2) -- (2+8.3-0.6,-3.+2);
\draw [->] (2+8.3,-3.+2) -- (2+9.3-0.6,-3.+2);
\draw [->] (2+9.3,-3.+2) -- (2+10.3-0.6,-3.+2);
\draw [->] (2+7.3,-2.+2) -- (2+8.3-0.6,-2.+2);
\draw [->] (2+8.3,-2.+2) -- (2+9.3-0.6,-2.+2);
\draw [->] (2+9.3,-2.+2) -- (2+10.3-0.6,-2.+2);
\draw [->] (2+7.3,-1.+2) -- (2+8.3-0.6,-1.+2);
\draw [->] (2+8.3,-1.+2) -- (2+9.3-0.6,-1.+2);
\draw [->] (2+9.3,-1.+2) -- (2+10.3-0.6,-1.+2);

\begin{scriptsize}{\tiny}

\node(2+A+2) at (2+4.,-6.+2) {$1246$};
\node(2+B+2) at (2+4.,-5.+2) {$1247$};
\node(2+C+2) at (2+4.,-4.+2) {$1257$};
\node(2+D+2) at (2+4.,-3.+2) {$2357$};
\node(2+E+2) at (2+5.,-6.+2) {$1346$};
\node(2+F+2) at (2+5.,-5.+2) {$1347$};
\node(2+G+2) at (2+5.,-4.+2) {$1457$};
\node(2+H+2) at (2+5.,-3.+2) {$2457$};
\node(2+I+2) at (2+6.,-6.+2) {$1356$};
\node(2+J+2) at (2+6.,-5.+2) {$1367$ };
\node(2+K+2) at (2+6.,-4.+2) {$1467$ };
\node(2+L+2) at (2+6.,-3.+2) {$2467$};
\node(2+m+2) at (2+7.,-6.+2) {$8135$};
\node(2+n+2) at (2+7.,-5.+2) {$8136$};
\node(2+o+2) at (2+7.,-4.+2) {$8146$};
\node(2+p+2) at (2+7.,-3.+2) {$1246$ };
\node(2+q+2) at (2+7.,-2.+2) {$1247$ };
\node(2+r+2) at (2+7.,-1.+2) {$1257$ };
\node(2+s+2) at (2+8.,-3.+2) {$1346$ };
\node(2+t+2) at (2+8.,-2.+2) {$1347$};
\node(2+u+2) at (2+8.,-1.+2) {$1457$ };
\node(2+v+2) at (2+9.,-3.+2) {$1356$};
\node(2+w+2) at (2+9.,-2.+2) {$1367$ };
\node(2+x+2) at (2+9.,-1.+2) {$1467$ };
\node(2+y+2) at (2+10.,-3.+2) {$8135$};
\node(2+aa+2) at (2+10.,-2.+2) {$8136$};
\node(2+Aa+2) at (2+10.,-1.+2) {$8146$};
\end{scriptsize}
\end{tikzpicture}$$

Then the unique $I \in \mathcal{C}_{4}(\mathcal{T})$ such that $\widehat{I}=135$ is $I=1356$. After mutating at this vertex, the quiver obtained is as follows.
\[
\begin{tikzpicture}
\node(00) at (0,0){1246};
\node(20) at (4,0){1346};
\node(40) at (8,0){2467};
\node(02) at (0,2){1247};
\node(22) at (4,2){1347};
\node(42) at (8,2){1367};
\node(04) at (0,4){1257};
\node(24) at (4,4){1457};
\node(44) at (8,4){1467};

\draw (00) edge[->] (02);
\draw (02) edge[->] (04);
\draw (20) edge[->] (22);
\draw (22) edge[->] (24);
\draw (44) edge[->, bend left] (40);
\draw (42) edge[->] (44);

\draw (00) edge[->] (20);
\draw (02) edge[->] (22);
\draw (22) edge[->] (42);
\draw (04) edge[->] (24);
\draw (24) edge[->] (44);
\draw (40) edge[->, bend left] (00);
\end{tikzpicture}
\]
\end{example}

\bibliographystyle{alpha}
\bibliography{extending}

\end{document}